\documentclass[11pt]{article}
\usepackage{amsmath,amssymb,amscd,amsthm,graphicx,enumerate,enumitem,framed}
\usepackage[a4paper]{geometry}
\usepackage{hyperref}

\title{Mean Curvature Flow of Arbitrary Co-Dimensional Reifenberg Sets}
\author{Or Hershkovits\thanks{The author was partially supported by NSF grants DMS 1406407 and DMS 1105656}}
\date{\today}

\usepackage{amsfonts}

\numberwithin{equation}{section}
  \theoremstyle{plain}
 \newtheorem{theorem}[equation]{Theorem}

\newtheorem{corollary}[equation]{Corollary}
 \newtheorem{lemma}[equation]{Lemma}
 
 \theoremstyle{remark}
 \newtheorem{remark}[equation]{Remark}
 \theoremstyle{remark}
\theoremstyle{definition}
\newtheorem{definition}[equation]{Definition}
\newtheorem*{property_a}{Property A}
\newtheorem*{property_b}{Property B}
\newtheorem*{property_c}{Property C}

\newcommand*{\rom}[1]{\expandafter\@slowromancap\romannumeral #1@}

\newcommand{\de}{\varepsilon}

\newcommand{\spa}{\mathrm{span}}

\begin{document}
\maketitle

\begin{abstract}
We study the existence and uniqueness of smooth mean curvature flow, in arbitrary dimension and co-dimension, emanating from so called  $k$-dimensional $(\de,R)$ Reifenberg flat sets in $\mathbb{R}^n$. Our results generalize the ones from \cite{Her}, in which the co-dimension one case (i.e. $k=n-1$) was studied. For $\de$ fixed, this class is general enough to include (i) all $C^2$ sub-manifolds (ii) all Lipschitz sub-manifolds with Lipschitz constant less than $\de$ (iii) some sets with Hausdorff dimension larger than $k$. The Reifenberg condition, roughly speaking, says that the set has a weak metric notion of a $k$-dimensional tangent plane at every point and scale, but those tangents are allowed to tilt as the scales vary. We show that if the Reifenberg parameter $\de$ is small enough, the (arbitrary co-dimensional) level set flow (in the sense of Ambrosio-Soner \cite{AS}) is non fattening, smooth and attains the initial value in the Hausdorff sense. In particular, our result generalizes the one from \cite{Wan} and, in fact, all  known existence and uniqueness results for smooth mean curvature flow in arbitrary co-dimension. The largest deviation from the proof in \cite{Her} comes in the proof of uniqueness (i.e. non-fattening), where one is forced to work with the viscosity notion of the high co-dimensional level set flow from \cite{AS}, rather than Ilmanen's more geometric definition \cite{Ilm3}. This study leads to a general (short time) smooth uniqueness result, generalizing the one for evolution of smooth sub-manifolds, which may be of independent interest, even in co-dimension one.   
\end{abstract}

\section{introduction}
For a $k$-dimensional manifold  $M^k$, a family of smooth embeddings  $\phi_t:M^{k} \rightarrow \mathbb{R}^{n}$ for $t\in (a,b)$ is said to evolve by mean curvature if it satisfies the equation $\frac{d}{dt}\phi_t(x)=\vec{H}(\phi_t(x))$, where $\vec{H}$ is the mean curvature vector. If a compact submanifold  $M \subseteq \mathbb{R}^{n}$ is of type $C^2$, it follows from standard parabolic PDE theory that there exists a unique mean curvature flow starting from $M$ for some finite maximal time $T$. \\

The question of mean curvature flow (and geometric flows in general) with rough initial data, i.e. when the $C^2$ assumption is weakened, has been researched extensively (see e.g \cite{EH2},\cite{EH},\cite{Wan},\cite{Sim2},\cite{KL},\cite{Lau},\cite{Her}). For the co-dimension one, arbitrary dimensional mean curvature flow, two results form the forefront in that regard: In the case that $M$ is merely Lipschitz, short time existence was proved by Ecker and Huisken in the celebrated paper \cite{EH}. More recently, under the assumption that the initial set is $(n-1)$-dimensional $(\de,R)$  Reifenberg flat (see Definition \ref{Reif_def} ) with $\de$ sufficiently small, short time existence and uniqueness was shown in \cite{Her} (see also Theorem \ref{main_thm}). Those two results are very different in nature;  The result in \cite{EH} allows \textit{any} Lipschitz submanifold as an input, by that also dictating a local graph structure and a finite local area. The result in \cite{Her} allows some \textit{higher Hausdorff dimensional} sets which are not graphical at any scale (such as variants of the Koch-snowflake) to be taken as inputs, but it requires $\de$ to be \textit{small}. Note however that the Lipschitz assumption implies the Reifenberg property, the Reifenberg parameter $\de$ being the Lipschitz constant.\\

In the high co-dimensional case, the optimal known result, which is due to Wang, speaks about the same objects as Ecker-Huisken's result, but has the smallness character of the result in \cite{Her}. More precisely, it was shown in \cite{Wan} that there exists some $\de_0$ such that if $M$ is uniformly locally Lipschitz $k$-dimensional sub-manifold of $\mathbb{R}^n$, with Lipschitz constant less than $\de_0$ (i.e. there exists some $R>0$ such that  every point has a ball of radius $R$ around it on which the sub-manifold is an $\de_0$-Lipschitz graph) then there exists a mean curvature flow emanating from it (in light of the example in \cite{OL}, the smallness assumption in high co-dimension is necessary). By the discussion above, the high co-dimensional generalization of the result in \cite{Her}, which will be stated shortly, will form a full (qualitative) generalization of the result in \cite{Wan}. To state this result, we first need to define the objects to which it applies.            

\begin{definition}[\cite{Reif}]\label{Reif_def}
Given $\de>0$, $R>0$ and $k\in \mathbb{N}$, a compact connected set $X \subseteq \mathbb{R}^{n}$ is called \textbf{$k$-dimensional $(\de,R)$ Reifenberg flat} if for every $x\in X$ and $r<R$ there exists a $k$-dimensional plane $P_{x,r}$ passing through $x$ such that
\begin{equation}
d_H(B(x,r)\cap P_{x,r}, B(x,r)\cap X) \leq \de r.
\end{equation}  
Here $d_H$ is the Hausdorff distance.
\end{definition}
Any $C^2$ $k$-submanifold is easily seen to be $k$-dimensional $(\de,R)$ Reifenberg flat for some $\de,R>0$.  Every uniformly locally Lipschitz $k$-submanifold of $\mathbb{R}^n$ is trivially $k$-dimensional $(\de,R)$ Reifenberg flat as well. The Reifenberg condition is however general enough to include some sets with Hausdorff dimension larger than $k$ (see \cite{Tor},\cite{Her}). \\

Another notion that one needs in order to discuss evolution of non smooth initial data  is that of a weak solution to the $k$-dimensional mean curvature flow. This weak mean curvature flow is called the level set flow, as its original definition, due to Evans and Spruck \cite{ES1} and Chen-Giga-Goto \cite{CGG} (in the co-dimension one case) was via viscosity solutions for the equation of a level set of a function evolving by mean curvature.  In co-dimension one, a geometric, avoidance principle based, equivalent definition was given in \cite{Ilm3}. In high co-dimension, while a viscosity solution based definition was given in \cite{AS}, there is no effective geometric definition of weak mean curvature flow (of arbitrary sets), as even smooth flows cease to satisfy avoidance. The definition of the level set flow in \cite{AS} is technical, and will thus be postponed to Section \ref{uniq_sec}. For now, it suffices to know that the $k$-dimensional level set flow is a semi-group action of $\mathbb{R}_+$ (time) on compact sets $X\subseteq \mathbb{R}^n$, $(t,X)\mapsto X_t$ which, starting from an initial $k$-dimensional submanifold, coincides with smooth $k$-dimensional mean curvature flow, for as long as the latter is defined. Up to the specificities of this association, which will defined formally (and further investigated) in Section \ref{uniq_sec}, we can now state our main theorem. 

\begin{theorem}\label{main_thm}
There exists some $\de_0,c_0>0$ such that if $X$ is $k$-dimensional $(\de,R)$-Reifenberg flat in $\mathbb{R}^n$ for $0<\de<\de_0$ then the $k$-dimensional level set flow (in the sense of \cite{AS}) emanating from $X$ $(X_t)_{t\in (0,c_0R^2)}$ is a smooth $k$-dimensional  mean curvature flow, which further attains the initial value $X$ in the following sense: 
\begin{equation}
\lim_{t \rightarrow 0}d_{H}(X,X_t)=0.
\end{equation}
In fact, there exist some $c_1,c_2>0$ with $c_1^2<\frac{1}{8}$ and $\frac{1}{4c_1}-c_2>\sqrt{2k}$ such that the following estimates on the evolution hold:
\begin{enumerate}[label=\roman*.]
\item $|A(t)| \leq \frac{c_1}{\sqrt{t}}$.
\item $d_H(X_t,X) \leq c_2\sqrt{t}$.
\item $X_t$ has a tubular neighborhood of size $\frac{\sqrt{t}}{4c_1}$.
\end{enumerate}
\end{theorem}
\begin{remark}
Recall that level set-flow should be thought of as (and in some regards is)  ``the set of all possible evolutions'' (see \cite[Sec. 10]{Ilm} and \cite[Thm. 5.4]{AS}). Thus, in addition to existence of a smooth mean curvature flow emanating from $k$-dimensional Reifenberg sufficiently flat sets, we get uniqueness in the strongest possible sense.
\end{remark}
\begin{remark}
 In light of the discussion preceding the statement of the theorem, this qualitatively generalizes the result from \cite{Wan}. As the $\de_0$ of Theorem \ref{main_thm}  is smaller than the one from \cite{Wan}, the generalization is only qualitative, i.e. there are still initial submanifolds for which the result in \cite{Wan} is applicable while Theorem \ref{main_thm} is not.     
\end{remark}
\vspace{5 mm}

The proof of theorem \ref{main_thm} is naturally divided into two parts: existence,  in which we will construct a mean curvature flow satisfying estimates $i-iii$  of  Theorem \ref{main_thm} and uniqueness, in which we will show that the resulted flow is actually the level-set flow. The proof of the existence part  goes along the same lines as the existence part of the main theorem in \cite{Her} and will be described in Section \ref{exist_sec}, where the parts that are identical will only be stated, referring to \cite{Her} for the proof. The parts which require additional work to the one done in the co-dimension one case will be treated in full.  The proof of the uniqueness, which will inhibit Section \ref{uniq_sec},  is completely different from the one in \cite{Her}. Back there, one could work almost entirely in the realm of smooth solution, utilizing inward and outward approximations \cite[Cor. 2.11]{Her} and Ilmanen's avoidance based definition of the level set flow \cite{Ilm3}. The uniform estimates on the evolution of the approximations (see \cite[Thm. 1.14]{Her} and Theorem \ref{un_est}) coupled with a separation estimate \cite[Thm. 1.20]{Her} were then used to show that the flows emanating from the inward and outward approximations remain very close, providing barriers to the level set flow. In high co-dimension there is no notion of ``inside'' and ``outside'', and, as discussed above,  there is no avoidance based definition of the level set flow. In order to show uniqueness (and in fact, even define it) we will therefore have to revert to work entirely in the viscosity solution realm.  While in co-dimension one the viscosity solution definition of the level set flow  is very well known and standard, the high co-dimensional analogue of it, introduced in \cite{AS} is far less known. Some part of our work will consist of exploring it a bit further than what was done in \cite{AS}.  

The short time uniqueness of the flow is an immediate corollary of the existence part of Theorem \ref{main_thm} and the following general strong smooth uniqueness criterion for the level set flow, which is of interest by its own right.

\begin{theorem}\label{smooth_uniq}
Let $X\subset \mathbb{R}^n$ be a connected compact set, let $c_1,c_2>0$ be constant  and let $(X_t)_{t\in(0,T]}$ be a smooth $k$-dimensional mean curvature flow, satisfying
\begin{enumerate}[label=\roman*.]
\item $|A(t)| \leq \frac{c_1}{\sqrt{t}}$.
\item $d_H(X_t,X) \leq c_2\sqrt{t}$.
\item $X_t$ has a tubular neighborhood of size $\frac{\sqrt{t}}{4c_1}$.
\end{enumerate}
Assume further that $c_1^2\leq\frac{1}{8}$ and $\frac{1}{4c_1}-c_2>\sqrt{2k}$. Then $X_t$ is the level set flow of $X$.
\end{theorem} 

Theorem \ref{smooth_uniq} is a quantitative generalization of the fact that, starting from a smooth sub-manifold,  the level set flow coincides with smooth mean curvature flow.  As in the smooth case (see \cite[Sec. 3]{AS}), the idea is to use the distance to $X_t$ to construct non-negative lower barriers to the level set equation. As conditions $i-iii$ are more precise than a smoothness  assumption, this choice should be made more quantitatively; The construction of the sub-solution, as well as the choice of the parabolic neighborhood along the boundary of which the barrier can be seen to be smaller than the solution, should reflect estimates $i-iii$. In terms of the proof, this means that as opposed to the smooth case, where one could get along by  some continuity based arguments,  in our case we will be forced to compute some quantities more explicitly, and to use avoidance of balls (which is true in arbitrary co-dimension (see Lemma \ref{lower_bound})) to get some initial estimates on the behavior of the level set equation (with the right initial data).                  
\begin{remark}
As a flow satisfying $i-iii$ provides approximations at different scales to $X$, having such smooth flow implies some regularity of $X$. This regularity is far weaker than the one assumed in Theorem \ref{main_thm} (c.f. Theorem \ref{approx_smfd}), so we expect Theorem \ref{smooth_uniq} to be applicable in other situations as well.
\end{remark}
 
\section{Existence}\label{exist_sec}
The proof of the existence of a flow $(X_t)_{t\in (0,c_0R^2)}$ satisfying estimates $i-iii$ of Theorem \ref{main_thm}  is divided,  as in the co-dimension one case of \cite{Her},  into three steps: construction of approximations at different scales (Section \ref{approx_sec}) , obtaining uniform estimates on the flows emanating from them, and passage to a limit (Section \ref{conclus_sec}). The proof of the uniform estimates in turn, consists of three major ingredients: estimates for graphical mean curvature with small initial data on thick cylinders (Section \ref{main_estimate_sec}), an interpolation lemma (Section \ref{inter_ext_sec}), and an iteration scheme (Section \ref{conclus_sec}). As stated in the introduction, in this section we will address in full the parts that require different treatment than the one in \cite{Her} and mention the parts that remain the same.

\subsection{Approximation}\label{approx_sec}

A guideline to proving estimates on a class of weak objects is to first approximate them by smooth objects, then derive estimates that depend only on quantities that are expressible for the weak objects as well, and finally pass to a limit. The first step in our case is the following approximation theorem, which is essentially from \cite[Appendix B]{KL1}, where the hypothesis used are different, but the construction remains the same. 

\begin{theorem}[\cite{KL1}]\label{approx_smfd}      
For every $\delta>0$ there exists $\de>0$ such that if $X\subseteq \mathbb{R}^{n}$ is $k$-dimensional $(\de,R)$ Reifenberg flat, then for every $0<r<R/10$ there exists a $k$ dimensional sub-manifold $X^r$ such that 
\begin{enumerate}
\item $d_H(X,X^r) \leq \delta r$.
\item $|A^r| \leq \frac{\delta}{r}$
\item For every $x\in X$, $X^r\cap B(x,r)=G\cup B$ where $G$ is connected and $B \subseteq B(x,r)-B(x,(1-\delta r))$.
\end{enumerate}
\end{theorem}

\begin{remark}
In \cite{Her} we utilized a stronger global approximation result from \cite{HW}, but its proof depended on mollifying the characteristic function of the domain bounded by $X$, which only works in co-dimension one.   
\end{remark}
\begin{remark}
Reifenberg's topological disk theorem \cite{Reif} follows easily from Theorem \ref{approx_smfd}; The approximations at comparable scales are graphical above one another, and composing those graphical representations yields the bi-H\"{o}lder parametrization (c.f. proof Corollary \ref{approx_smfd_add}).
\end{remark}
\begin{remark}\label{almost_same_tangents}
While the ``approximate tangents'' of a $k$-dimensional $(\de,R)$ Reifenberg flat set vary with point and scales, comparable scales and nearby points have very close approximating tangents. More precisely, for every $\delta>0$ there exists $\de>0$ such that if $X$ is $k$-dimensional $(\de,R)$ Reifenberg flat, then for every $r<R/10$ and $x_1,x_2\in X$ with $d(x_1,x_2)<r$,  both $||P_{x_1,r}-P_{x_2,r}||<\delta$ and $||P_{x_1,r}-P_{x_1,10r}||<\delta$.  Here $P_{x_1,r}$ is as in Definition \ref{Reif_def},  $||-||$ is the operator norm and we use the standard identification of a $k$ sub-space with the projection operator to it. This elementary observation is the key property of Reifenberg flat sets that leads to Theorem \ref{approx_smfd}.  
\end{remark}
\begin{proof}[Proof of Thm. \ref{approx_smfd}]
 Let $\phi:[0,\infty)\rightarrow \mathbb{R}_+$ be a smooth function such that $\phi|_{[0,1]}=1$ and $\phi|_{[2,\infty)}=0$. Let $G(n-k,n)$ be the $n-k$ dimensional Grassmanian and $E(n-k,n)$ be the total space of the tautological vector bundle over $G(n-k,m)$. Now, fix $r<R/10$ and let $p_1,\ldots, p_L\in X$ be a maximal collection of points such that the balls $B(p_i,r/6)$ are disjoint. In particular $N(X,r):=\{y\in \mathbb{R}^n\;\;\mathrm{s.t}\;\; d(y,X)<r\} \subseteq \bigcup B(p_i,2r)$. For every $i\in \{1,\ldots L\}$ fix an $n-k$ dimensional plane $Q_i=P_{p_i,r}^{\perp}$ where $P_{p_i,r}$ is as in the definition of  a Reifenberg flat set, and set $\phi_i(y)=\phi(|y-p_i|/2r)$. Now, for every $y\in N(X,r)$ define $O_y=\frac{\sum_{i=1}^{L}\phi_i(y)Q_i}{\sum_{i=1}^{L}\phi_{i}(y)}$ and note that there exists some $I$, independent of $X$ and $r$, such that for every $y\in N(X,r)$, at most $I$ of the summands in the numerator are non zero. Fixing $x\in X$ and letting $Q_x=P_{x,r}^{\perp}$  we note that for every $y\in B(x,2r)$, all the contributors to $O_y$ are in $B(x,6r)$. Thus, by Remark \ref{almost_same_tangents}, $||O_y-Q_x||_{C^3(B(x,2r))}\leq \alpha(\de)$, where $\lim_{\de\rightarrow 0}\alpha(\de)=0$. As $O_y$ is symmetric and very close to $Q_x$ , If we let $\tilde{Q}_y$ be the orthonormal projection to the span of the $n-k$ eigenvectors of $O_y$ with eigenvalues close to $1$, we also have 
\begin{equation}\label{closeness}
||\tilde{Q}_y-Q_x||_{C^3(B(x,2r))}\leq \alpha(\de).
\end{equation}
Finally, set $\eta(y)=\frac{\sum_{i=1}^{L}\phi_i(y)\tilde{Q}_y(y-p_i)}{\sum_{i=1}^{L}\phi_{i}(y)}$ and define $\pi:N(X,r)\rightarrow E(n-k,n)$ by $\pi(y)=(\pi_1(y),\pi_2(y))=(\tilde{Q}_y,\eta(y))$ and define $X^{r}$ to be the inverse image of the zero section $\xi$ in $E(n-k,n)$.

\vspace{5 mm}

Let us verify that $X^{r}$ indeed satisfies the desired properties. First observe that if $\pi_2(y)=0$ then $y\in N(X,\delta(\de))$ where $\lim_{\de \rightarrow 0}\delta(\de)=0$. Indeed, let $x\in X$ be the closest point to $y$ and observe, as before, that if $p_i$ provides a non zero contribution to $O_y$ then $p_i\in B(x,6r)$ and so by the Reifenberg property at scales $r$ and $10r$ $|Q_x(p_i)|<\delta(\de)$ and so $|Q_x(y)|<\delta(\de)$ by \eqref{closeness}. Thus $d(P_{x,r}(y),y)<\delta(\de)$ and by the Reifenberg property $d(X,y)<\delta(\de)$. Moreover, as $\pi(N(X,r))\pitchfork \xi$ (again, by \eqref{closeness}), $X^{r}$ is a $k$-dimensional sub-manifold. Fixing $x\in X$ we have that $|\tilde{Q}_x(x)|<\delta(\de)$ and so by \eqref{closeness} again, there is $y\in N(x,r)$ with $d(y,x) < \delta(\de)$ such that $\pi_2(y)=0$. Thus, we have established that $X^r$ is a sub-manifold that satisfy condition 1.

Taking $x\in X$ and $x'\in P_{x,r}\cap(B(x,r))$ we see that $||\pi_{2}|_{x'-x+Q_x}-Id||_{C^3(B(x',3r))}<\delta(\de)$ and so by the quantitative version of the the inverse function theorem, there exists a unique point $y\in (x'-x+Q_x)\cap B(x,2r)$ with $\pi_2(y)=0$. Thus $X^{r}\cap P_{x,r}^{-1}(B^{k}(x,r))\cap B(x,2r)$ is a graph of a function $f$ over $P_{x,r}\cap B^{k}(x,r)$ with $||f(x_1,\ldots, x_k)||_{C^3(B^k(x,r))}=||Q_x(x_1,\ldots, x_k, f(x_1,\ldots, x_k))||_{C^3(B^k(x,r))}<\delta(\de)$. This completes the proof.    
\end{proof}

\begin{corollary}\label{approx_smfd_add}
For every $\delta>0$ there exists $\de>0$ such that if $X\subseteq \mathbb{R}^n$ is $k$-dimensional  $(\de,R)$ Reifenberg flat set then in addition to $1-3$ of Theorem \ref{approx_smfd}, for every $x\in X$ and $s\in (r,R/10)$, $X^r\cap B(x,s)$ can be expressed as $X^r\cap B(x,s)=G\cup B$ where $G$ is connected and $B\subseteq B(x,s)-B(x,(1-5\delta)s)$.
\end{corollary}
\begin{proof}
Just like in \cite[Lemma 4.4 and Lemma 4.9]{Her} (see also Lemma \ref{interpolation}), conditions (1)-(3) of Theorem \ref{approx_smfd} imply that for $\de>0$ sufficiently small, $X^s$ has a tubular neighborhood of radius $s/4$ and that $X^{s/4}$ is a graph of a function $f_s$  over $X^s$ with $|f_s(x)| \leq 2\delta s$. Defining $f:X^s\rightarrow X^{4^{-j}s}$ by $f(y)=f_{4^{-k+1}s}\circ f_{4^{-k+2}s} \circ \ldots \circ f_s(y)$ we see that for every $y\in X^s$, $d(f(y),y)< 4\delta s$. This, together with property (3) for $X^s$ completes the proof.
\end{proof}

\subsection{Estimates for Graphical Mean Curvature Flow with Small Initial Data on Thick Cylinder}\label{main_estimate_sec}
In this section, we will generalize the proof of the main estimate for graphical mean curvature flow   \cite[Thm. 5.1]{Her} to the arbitrary co-dimensional case.

\begin{theorem}\label{main_est}
There exist some $c>0$ such that for every $\delta>0$ and $M>0$, there exist positive $\tau_0=\tau_0(M,\delta)<<1$ and $\lambda_0=\lambda_0(M,\delta)<<1$ such that for every $0<\lambda<\lambda_0$ there exists  some  $\de_0=\de_0(\delta,M,\lambda)$ such that for every $0<\tau<\tau_0$  and $\de<\de_0$ the following holds:

\noindent If $u:B^k(p,r)\times [0,\tau r^2]\rightarrow \mathbb{R}^{n-k}$ is a graph moving by mean curvature such that:
\begin{enumerate}
\item For every $(x,t) \in B(p,r)\times [0,\tau r^2]$ 
\begin{equation}
|\nabla u(x,t)| \leq M\de,\;\;\;\;\;|u(x,t)| \leq M^2\beta.  
\end{equation}
\item  For every $x \in B(p,r)$ we have
\begin{equation}
|\nabla u(x,\lambda \tau r^2)| \leq \de,\;\;\;\;\; |u(x,\lambda \tau r^2)| \leq \beta.
\end{equation}
\end{enumerate}
Then:
\begin{equation} 
|A(p,\tau r^2)| \leq (1+\delta)\frac{1}{\sqrt{\pi}}\frac{\de}{\sqrt{\tau} r},\;\;\;\;\;|A(p,\tau r^2)| \leq c\frac{\beta}{\tau r^2}+\delta\frac{\de}{\sqrt{\tau} r}.
\end{equation}
\end{theorem}
As in the proof of the co-dimension one case, the idea is to regard the graphical mean curvature flow equation as a non-homogeneous heat equation. The controlled growth of the function and its gradient (condition 1 in Theorem \ref{main_est}),  and the thickness of the cylinder ($\tau(M)<<1$) allows one to derive a Schauder type estimate for the non-homogeneous heat equation \cite[Thm. 5.12]{Her} for which the homogeneous part ``does not see the boundary'' and depends only on the initial slice (i.e. on $\de,\beta$, but not on $M$). In that regard, the estimate one gets for the homogeneous part are (up to a multiplicative constant) like the ones obtained for physical solutions to the heat equation on the full space.  Proving Theorem \ref{main_est} then reduces to showing that  H\"{o}lder norm of the non-linearity behaves sub-linearly. 

In the co-dimension one case, the major step towards obtaining those estimates was a H\"{o}lder gradient estimate for $u$, which was proved by tracing the dependences of the constants in the proof of H\"{o}lder gradient regularity for parabolic quasilinear equations of general type \cite{Lie}. This led to showing that some H\"{o}lder norm of $\nabla u$ is at most linear in the $C^0$ norm of $\nabla u$ (when the latter is small). Such argument is not valid in the high co-dimensional case (as there is no such general H\"{o}lder gradient estimate), but by virtue of a compactness argument, one gets a weaker result, which will nevertheless suffice for our purposes.

\begin{theorem}\label{holder_grad_bds}
There exists some $\tau_0>0$ such that for every $\delta>0$ there exists $\de>0$ such that for every $0<\tau<\tau_0$ the following holds: If $u:B^k(p,r)\times [0,\tau r^2]\rightarrow \mathbb{R}^{n-k}$  is a solution to the graphical mean curvature flow with $||Du||<\de$ then, setting $B^{\tau}(p,r)=B(p,(1-1000\sqrt{\tau_0})r)\times [0,\tau r^2]$
\begin{equation}
\sup_{z_1,z_2\in B^{\tau}(p,r)}d_{z_1,z_2}^\alpha\frac{||Du(z_1)-Du(z_2)||}{d(z_1,z_2)^{\alpha}}<\delta,\;\;\;\;\;\; \sup_{z\in B^{\tau}(p,r)}d_{z}\frac{||D^2u(z)||}{d(z_1,z_2)}<\delta.
\end{equation}
Here
\begin{equation}
d((x_1,t_1),(x_2,t_2))=\sqrt{|x_1-x_2|^2+|t_1-t_2|},
\end{equation}
$d_{z_1}=d(z_1,\partial(B(p,r)\times [0,\tau r^2]))$ (note that this is \underline{not} the distance to the boundary of $B^{\tau}(p,r)$) and $d_{z_1,z_2}=\min(d_{z_1},d_{z_2})$.
\end{theorem}
\begin{proof}
We argue by contradiction. Assume w.l.g that $r=1$ and note that the first inequality is trivial when $d(z_1,z_2) \geq d_{z_1,z_2}$ and follows by integration from the second inequality otherwise. Suppose that there exist some $\delta$ such that for every $m$ we can find $\tau^m<\tau_0$ and a solution of the graphical MCF $u^m:B^k(x,r)\times [0,\tau^m r^2]\rightarrow \mathbb{R}^{n-k}$ with $||Du^m||<1/m$ and $z_m \in B^{\tau^m}(p,r)$ such that 
\begin{equation}
d_{z_m}||D^2u^m(z_m)||\geq\delta.
\end{equation}
Setting $z_m=(x_m,t_m)$, the closest boundary point to $z_m$ is $(x_m,0)$. Let $\lambda_m=\sqrt{\frac{\tau_0}{t_m}}\geq 1$ and define 
\begin{equation}
v^m(y,s)=\lambda_m\left[u^m(x_m+y/\lambda_m,s/\lambda_m^2)-u^m(x_m,0)\right].
\end{equation} 
Note that $v^2_m(0,0)=0$ and set $\xi_m=(y_m,s_m)=(0,\tau_0)$ . By the definition of $B^{\tau^m}(p,r)$, we conclude that $v^m$ is defined (at least) on $B(0,1000\sqrt{\tau_0})\times [0,\tau_0]$, satisfies the graphical mean curvature flow equation, and while $||Dv^m||\leq 1/m$, for $\xi_m =(0,\tau_0)$ we have  
\begin{equation}
||D^2v^m(\xi^m)||\geq \delta/\sqrt{\tau_0}.
\end{equation}
On the other hand, by the estimate from \cite{Wan}, the sequence sub-converges to a solution of the graphical mean curvature flow which is on one hand constant and on the other, has non-vanishing second derivative.  
\end{proof}  

Now, the graphical mean curvature equation has the form
\begin{equation}
\partial_t u -\Delta u=a^{ij}\frac{\partial^2 u}{\partial x_i\partial x_j}=N(Du,D^2u)
\end{equation}
where
\begin{equation}
a^{ij}=\left[\left(\delta_{kl}+\left<\frac{\partial u}{\partial x_k},\frac{\partial u}{\partial x_l}\right>\right)^{-1}\right]^{ij}-\delta^{ij}.
\end{equation}
Note that $a^{ij}$ is a rational function in the gradient of $u$, where the numerator $P(Du)$ has neither free coefficients, nor terms that are linear in $Du$. Thus, by Theorem \ref{holder_grad_bds} (and the estimates from \cite{Wan}), for every $\tau<\tau_0$ and every $\delta$ there exists $\de_0>0$ such that for every $\tau<\tau_0$, for every solution of the graphical MCF on $B(p,r)\times [0,\tau r^2]$ with $||Du||<\de$ we have, for every $x,y\in B^{\tau}(p,r)$
\begin{equation}
d_x|N(x)| \leq \delta \de,\;\;\;\;\; d_{x,y}^{1+\alpha}\frac{|N(x)-N(y)|}{d(x,y)^{\alpha}} \leq \delta \de.
\end{equation}
As discussed above, those estimates for the non-linearity, together with\cite[Thm. 5.12]{Her} imply Theorem \ref{main_est}.
  
\subsection{Extension and Interpolation}\label{inter_ext_sec}
In this section we include two very simple result. The first regards the extension of curvature bounds forward in time for arbitrary compact submanifolds, while the second is an interpolation result, at the presence of curvature bounds and Hausdorff bounds.

By \cite[Lemma 2.1]{Wan}, the evolution of the second fundamental form in arbitrary co-dimension satisfies the inequality
\begin{equation}
\frac{d}{dt}|A|^2 \leq \Delta |A|^2+ C(k,n)|A|^4.
\end{equation}
Thus, by the maximum principle, and the fact that the curvature must blow up at a singularity, one has the following extension lemma:
\begin{lemma}\label{extension}
If $M$ is a $k$-dimensional mean submanifold of $\mathbb{R}^n$ with $|A|\leq \alpha$ then there exist a mean curvature flow $M_t$ starting from $M$ that exists for (at least) $0\leq t \leq \frac{1}{C(k,n)\alpha^2}$, such that the norm of the second fundamental form satisfies the estimate
\begin{equation}
|A(t)| \leq \frac{\alpha}{\sqrt{1-C(k,n)\alpha^2t}}.
\end{equation}
\end{lemma}

The following elementary interpolation is central in our argument:
\begin{lemma}[Interpolation]\label{interpolation}
Fix $\delta>0,\alpha>0$. There exists $\beta_0>0$ such that for every $\beta<\beta_0$ the following holds: Assume $p\in X\subseteq B^n(p,r)$ where $X$ is a $k$-submanifold with
\begin{enumerate}
\item $|A| \leq \frac{\alpha}{r}$ .
\item $d_H(P\cap B^n(p,r),X \cap B^n(p,r)) \leq \beta r$ for $P=\spa\{e_1,\ldots,e_k\}$.
\end{enumerate}
 Then inside the cylinder $C_{\delta,\beta}=B^k(p,(1-\delta)r)\times [-\beta r,\beta r]^{n-k}$, the connected component of $p$ is a graph of a function $u$ over $P$  and we have the estimate
\begin{equation}
|\nabla{u}| \leq \sqrt{3\alpha\beta}
\end{equation}
(and $|u|\leq \beta r$).
\end{lemma}
\begin{proof}
Assume w.l.g. $r=1$ and $p=0$ and denote $Q=P^\perp$. For $\beta$ sufficiently small $C_{\delta/4,\beta}\subseteq B(0,1)$ and $\alpha\beta<1$. Now, let $x\in C_{\delta/2,\beta}$ and let $\gamma(t)$ be a unit speed geodesic with $\gamma(0)=p$. We may assume w.l.g., by possibly changing the parametrization according to  $t\mapsto -t$, that $\left<\gamma'(0),e_n\right>=\max_{v\in Q,\;||v||=1}\left<\gamma'(0),v\right>$ and that $x_n(\gamma(t)) \geq 0$. Letting $f(t)=x_n(\gamma(t))$ we find $f'(t)=\left<\gamma'(t),e_n\right>$ and $f''(t)=\left<\gamma''(t),e_n\right>=\left<\gamma''(t),e_n-\left<\gamma'(t),e_n\right>\gamma'(t)\right>\geq -\alpha\sqrt{1-f'(t)^2}$. The equality case of the above ODE for $f'(t)$ corresponds to a circle of radius $\frac{1}{\alpha}$. Letting $\mu(t):\mathbb{R}\rightarrow \mathbb{R}^2$ be a clockwise and unit speed parametrized circle of radius $\frac{1}{\alpha}$ with $\mu(0)=(0,0)$ and $\left<\mu'(0),e_2\right>=f'(0)$ we see that as long as $x_2(\mu(t))$ is increasing, and as long as $\gamma(t)\in C_{\delta/2,\beta}$,  $x_n(\gamma(t))\geq x_2(\mu(t))$. For $\beta$ sufficiently small (depending on $\alpha,\delta$) $x_2(\mu(t))$ will reach its maximum at time $0<T<\delta/4$ so the extra condition $\gamma(t)\in C_{\delta/2,\beta}$ is redundant . Thus $x_2(\mu(t))\leq x_n(\gamma(t))\leq \beta$, and an easy calculation for circles in the plane gives the bound
\begin{equation}\label{almost_tan}
\tan\angle(T_xX,P) \leq \frac{\sqrt{2\beta\alpha -\alpha^2\beta^2}}{1-\alpha\beta} \leq \sqrt{3\alpha\beta}
\end{equation}
for $\beta$ sufficiently small. 

What remains to be shown is that the connected component of $p$ is indeed a graph. Assume there exist $x_1,x_2\in X\cap C_{\delta,\beta}$ with $x_1\neq x_2$ but $P(x_1)=P(x_2)$. Observe that by \eqref{almost_tan}, $X\cap \overline{C_{\delta,\beta}}$ is a sub-manifold with boundary. Let $\gamma:[0,a]\rightarrow X \cap \overline{C_{\delta,\beta}}$ be a minimizing geodesic between $x_1$ and $x_2$. Such a geodesic is always $C^1$ and is smooth for as long $\gamma(t)$ is away from the boundary. For such $t$ however $||P(\gamma''(t))|| \leq \sqrt{3\alpha\beta}\alpha$ by \eqref{almost_tan} and so for $\beta$ sufficiently small, and as $\gamma'(0)$ is almost parallel to $P$,  $P(\gamma(t)))$ is almost a straight line until it hits the boundary (at some $t<4$). Since $\gamma(t)$ is $C^1$, and intersects the boundary with an exterior normal component, this is a contradiction.

To see that for every $y\in B^n(0,1-\delta)$ there is some $x\in X$ with $P(x)=y$, note that by the Hausdorff condition, we can find $\bar{x}\in X\cap B(0,(1-\delta/2))$ with $d(\bar{x},y)\leq \beta$ (when $\beta$ is small). Taking $\bar{y}=P(\bar{x})$ we see, again, by \eqref{almost_tan} for $\bar{x}$, and the fact that the curvature scale $\frac{1}{\alpha}$ is far bigger than $\beta$, that there will exist a point over $y$ as well.             
\end{proof}

\subsection{Construction of a Flow}\label{conclus_sec}
 For $\de$ sufficiently small, if $X$ is $k$-dimensional $(\de,R)$ Reifenberg flat,  Theorem \ref{approx_smfd} and Corollary \ref{approx_smfd_add} provide us with smooth approximations to $X$ at different scales. The extension lemma, Lemma \ref{extension}, the interpolation Lemma, Lemma \ref{interpolation}, and the a-priori estimate, Theorem \ref{main_est} can substitute the corresponding result of \cite{Her} in the iteration scheme of  \cite[Sec. 3.2,3.3]{Her}. Thus, just like there we obtain the following uniform estimates.

\begin{theorem}[Uniform estimates]\label{un_est}
 There exist some $\de$ and $c_0,c_1,c_2,c_3$ such that if $X$ is $k$-dimensional $(\de,R)$-Reifenberg flat, and considering the approximating surfaces $X^r$ from Theorem \ref{approx_smfd} and Corollary \ref{approx_smfd_add}, each $X^r$ flows smoothly by $k$-dimensional mean curvature for time  $t\in [0,c_0 R^2]$ and for every $t\in [c_3r^2,c_0R^2]$ we have: (1) $|A^r(t)| \leq \frac{c_1}{\sqrt{t}}$ where $A^r(t)$  the second fundamental form of $X^r_t$. (2) $d_H(X^{r}_t,X) \leq c_2\sqrt{t}$. (3)  For every $x\in X$ and  $s\in (\frac{\sqrt{t}}{c_1},R/4)$ we have
\begin{equation}\label{uni_est_conn}
B(x,s)\cap X^r_t=G\cup B
\end{equation}
where $G$ is connected and $B\cap B(x,\frac{9}{10}s)=\emptyset$. Moreover, the constants $c_1,c_2$ satisfy $c_1^2 \leq \frac{1}{8}$ and $c_1c_2<\frac{1}{2}$.
\end{theorem}
By the uniform estimates one can pass to a sub-limit flow $X_t$ as in Theorem \ref{main_thm}. Note that the condition $iii$ of Theorem \ref{main_thm} follows easily from $1-3$ of Theorem \ref{un_est} (see \cite[Lem 4.4]{Her}).

\section{Uniqueness}\label{uniq_sec}
In this section, we will prove Theorem \ref{smooth_uniq} which, together with the existence part of Theorem \ref{main_thm}, imply the full Theorem \ref{main_thm}. In section \ref{Prelim} we will recall the definition and some properties of the high co-dimensional level set flow from \cite{AS}. In section \ref{dist_evolve} we will recall and further explore the behavior of the associated level set operator on distance functions from smooth evolutions by mean curvature. Section \ref{smooth_uniq_sec} will be devoted to the proof of Theorem \ref{smooth_uniq}.

\subsection{Preliminaries}\label{Prelim}
Let us start by introducing some notations. For every $0\neq p\in \mathbb{R}^{n}$  define $P_p$ to be the projection to the orthogonal complement of $p$. Given an $n\times n$ symmetric matrix $A$ and such $p$, let $X=P_pAP_p$. If we denote the eigenvalues corresponding to $p^\perp$  by $\lambda_1(X) \leq\ldots \leq \lambda_{n-1}(X)$, define
\begin{equation}
F(p,A)=\sum_{i=1}^{k}\lambda_i(X). 
\end{equation} 

In \cite{AS}, the level sets of positive viscosity solution to the equation 
\begin{equation}\label{visc_eq}
\frac{d}{dt} u=F(\nabla u, \nabla^2 u)
\end{equation} 
were used to give a definition for weak mean curvature flow. Before diving into formalities, we will try to convince the reader that this approach is plausible. Let $(M_t)_{t\in[0,T]}$ be a smooth family of $k$-dimensional sub-manifolds of  $\mathbb{R}^n$. Let $u:\mathbb{R}^n\times [0,T]\rightarrow \mathbb{R}_+$ be a smooth function such that for every $t\in [0,T]$, $M_t=\{x\in \mathbb{R}^n \;\;\mathrm{s.t.}\;\; u(x,t)=0\}$ and such that $\nabla u \neq 0$ on a neighborhood of $M_t$, and $u\geq \delta_0>0$ outside that neighborhood. For $\de$ sufficiently small, $M^\de_t=\{x\in \mathbb{R}^n \;\;\mathrm{s.t.}\;\; u(x,t)=\de\}$ would be a smooth ``tubular'' hypersurface  around $M_t$. We would expect it to have $n-k-1$ principal curvatures that are very large, corresponding to ellipsoids in the orthogonal complement of $T_pM$ for $p\in M_t$. The other $k$ principal curvatures should be very close to the ones of $M_t$ w.r.t. the normal of $M^\de_t$, as for every geodesics $\gamma(s)$ of  $M_t$ and every point on $x\in M^\de_t$ which is closest to $\gamma(0)$,  there should be an almost geodesic curve in $M^\de_t$ which is orthogonal to the above mentioned ellipsoids which ``traces $\gamma$'' . The second fundamental form of $M^\de_t$ at $x\in M^\de_t$ is given by $A(x,t)=\frac{1}{|\nabla u |}P_{\nabla u(x,t)}\nabla^2u(x,t)P_{\nabla u(x,t)}$, so we expect $\frac{1}{||\nabla u||}F(\nabla u,\nabla^2u)$ to be very close to $-\vec{H}(\pi(x),t)\cdot \frac{\nabla u}{||\nabla u||}$, where $\vec{H}(\pi(x),t)$ is the mean curvature of $M_t$ at the point closest to $x$. On the other hand, the normal velocity of $M^\de_t$ at $x$ is given by $-u_t/||\nabla u||$ and so equation \eqref{visc_eq} tells us that, parameterizing the flow of $M_t^\de$ by $\phi^\de(x,t)$ and letting $\nu(x,t)$ be the normal to $M_t^\de$ at $x$, we have 
\begin{equation}
\frac{d}{dt}\phi^\de(x,t)\cdot \nu = \vec{H}(\pi(x),t)\cdot \nu+O(\de). 
\end{equation}
Thus, the entire ellipsoid around $\pi(x)$ moves approximately by $\vec{H}(\pi(x))$, which would correspond to a motion of $M_t$ by a $k$-dimensional mean curvature flow.\\

We proceed by defining the level-set flow and by collecting some of its properties  from \cite{AS}. Let $u_0$ be a non-negative, uniformly continuous function. Theorem 2.4 in \cite{AS} states that there exists a unique uniformly continuous, positive viscosity solution to equation \eqref{visc_eq} $u:\mathbb{R}^n\times [0,\infty)$ such that $u(-,0)=u_0$ (for the definition of viscosity solution, see \cite[Def. 2.1]{AS}). 
\begin{definition}\cite[Def 2.6]{AS}
Let $X\subseteq \mathbb{R}^n$ be a closed set. Let $u_0$ be a non-negative, uniformly continuous function such that $X=\{x\in \mathbb{R}^{n} \;\;\textrm{s.t.}\;\; u_0(x)=0\}$. Letting $u$ be the solution of the IVP of equation \eqref{visc_eq} with $u(x,0)=u_0(x)$, the $k$-dimensional \textbf{level-set flow} of $X$ is defined to be $X_t=\{x\in \mathbb{R}^{n} \;\;\textrm{s.t.}\;\; u(x,t)=0\}$.
\end{definition}
A-priori, this definition may depend on the choice of $u_0$, but Theorem 2.5 of \cite{AS} shows that different choices yield the same result.\\

The following three properties of the level-set flow and the level-set equation would be of importance to us.
\begin{property_a}\label{prop_a}
If $X$ is a smooth $k$-dimensional sub-manifolds of $\mathbb{R}^n$ and the smooth  mean curvature flow  $(X_t)_{t\in [0,T]}$ is defined for all $t\in [0,T]$ then $X_t$ is also the level-set flow of $X$ (see \cite[Cor. 3.9]{AS}). 
\end{property_a}
\begin{property_b}
If $u,v:\mathbb{R}^n\times [0,T]\rightarrow \mathbb{R}_+$ are two non-negative, uniformly continuous solutions to \eqref{visc_eq} then $||u-v||_{L^{\infty}(\mathbb{R}^n\times[0,T])} \leq ||u-v||_{L^{\infty}(\mathbb{R}^n\times\{0\})}$ (see \cite[Thm. 2.2]{AS}). 
\end{property_b}
\begin{property_c}\label{prop_c} 
Let $\Omega \subseteq \mathbb{R}^n \times [0,T]$ be a bounded domain, and let $v,u$ be non-negative, uniformly continuous functions on $\mathbb{R}^n\times [0,T]$ such that $v$ is a sub-solution to \eqref{visc_eq} in $\Omega$ and $u$ is a solution of \eqref{visc_eq} in $\Omega$. If $u\leq v$ on $\partial_{\mathrm{par}} \Omega$ the $u\leq v$ on $\Omega$. Here $\partial_{\mathrm{par}} \Omega$ is the parabolic boundary of $\Omega$ (see remark below).    
\end{property_c}
\begin{remark}
In \cite[Thm. 2.2]{AS} it was shown that $F$ satisfies the assumptions of \cite[Thm. 2.1]{GGIS}, and of Thm. 4.1 of \cite{CGG}, both of which imply Property C for domains of the form $\Omega=D\times [0,T]$ where $D\subseteq \mathbb{R}^n$ is a compact domain. The proof of Thm. 4.1 in \cite{CGG} works just as well for an arbitrary bounded domain $\Omega$, while compactness also gives Prop. 2.3 of \cite{GGIS}, from which point the proof of Thm. 2.1 of \cite{GGIS} works for arbitrary such $\Omega$ as well. This is, of course,  not surprising at all, as the weak maximum principle is a statement about the interior.       
\end{remark}

\subsection{Evolution of Distances}\label{dist_evolve}
Let $(M_t)_{t\in [0,T]}$ is a family of $k$-submanifolds evolving by smooth mean curvature flow. For each $0\leq t \leq T$, $M_t$ has a tubular neighborhood at which the distance function from $M_t$, $r_t$ is smooth. Studying the properties of $F(\nabla r_t,\nabla^2 r_t)$ and of $\frac{d}{dt}r_t-(\nabla r_t,\nabla^2 r_t)$ will occupy the first part of this section. In the second part, we will give a simple lower bound, which is based on avoidance of balls, for the solution of the level-set equation starting from a distance function from an \textit{arbitrary} set.\\

\begin{lemma}\label{dist_tub_stat}
Let $M$ be smooth, $k$-dimensional sub-manifold in $\mathbb{R}^{n}$ with normal injectivity radius $\rho$ and let $r$ be the distance function from $M$. For every $0\leq s \leq \rho$, let $M^s$ be the smooth co-dimension $1$ level set $M^s=\{y\in \mathbb{R}^{n}\;\;\mathrm{s.t}\;\;r(y)=s\}$ and let $A^s$ be the second fundamental form of $M^s$ with respect to the interior normal. Let $x\in M$, $\nu\in SN_xM$ (the unit normal space to $M$ at $x$), $P=T_xM$ and $Q=(P\oplus \nu)^\perp$. The following hold:
\begin{enumerate}
\item for every $0<s<\rho$ the principal directions of $M^s$ at $x+s\nu$ are independent of $s$ and split according to $P$ and $Q$.
\item $A^s|_{Q}=\frac{1}{s}Id$ while $A^s|_P$ is bounded inside the tubular neighborhood, and if $v_1,\ldots v_k\in P$ are the principal directions of  $A^{-\nu}(;)=\left<A(;), -\nu\right>$ with eigenvalues $\beta_1,\ldots, \beta_k$ then 
\begin{equation}\label{ev_prop}
A^s|_{P}(v_i,v_i)=\frac{\beta_i}{1+s\beta_i},
\end{equation} 
and $A^s|_P(v,v) <\frac{1}{s}$ for every $v\in P$ with $||v||=1$.
\end{enumerate}
\end{lemma}
\begin{remark}
Most of the above lemma is stated and proved in \cite[Thm 3.2]{AS}. Equation \eqref{ev_prop} was proved for some constants $\beta_i$, without identifying them as the principal directions of $-A(;)\cdot \nu$. This  fact  was not needed for the applications therein, and was stated there as a conjecture (which the authors did not care much about).        
\end{remark} 
\begin{proof}[Proof of Lemma. \ref{dist_tub}]
$(x+P^\perp)\cap M^s$ consist of a sphere of radius $s$ in $x+P^\perp$, whose normal at $x+s\nu$ is $-\nu$, so $A^s|_{Q}=\frac{1}{s}$ as stated. Fix $v=\sum a_iv_i\in P$, and let $\gamma(t)$ be a unit speed geodesic in $M$ with $\gamma(0)=x$ and $\dot{\gamma}(0)=v$, and let $\nu(t)$ be a normal field along $\gamma(t)$, which solves the linear system of ODEs  $\nu(0)=\nu$ and $N_{\gamma(t)}M(\dot{\nu}(t))=0$ in the normal bundle to $M$. In particular $||\nu(t)||=||\nu(0)||=1$, so $\mu(t):=\gamma(t)+s\nu(t)$ is a curve in $M^s$ such that
 \begin{equation}
||\dot{\mu}(0)||^2=||v+sA^{-\nu}(v,v_i)v_i||=\sum \left(a_i(1+s\beta_i)\right)^2
\end{equation}
and
\begin{equation}
A^s(\dot{\mu}(0),\dot{\mu}(0))=\left< \ddot{\mu}(0),-\nu\right>=A^{-\nu}(v,v)+s||\dot{\nu}(0)||^2=\sum a_i^2\beta_i+s\sum (a_i\beta_i)^2.
\end{equation}
As $s<\rho < \frac{1}{|\beta_i|}$, $1+s\beta_i>0$ and so
\begin{equation}
\frac{\sum a_i^2\beta_i+s\sum (a_i\beta_i)^2}{\sum \left(a_i(1+s\beta_i)\right)^2} <\frac{1}{s}
\end{equation}
holds. Additionally,
\begin{equation}
A^s(v_i,v_i)=\frac{\beta_i+s\beta_i^2}{(1+s\beta_i)^2}=\frac{\beta_i}{1+s\beta_i}.
\end{equation}
\end{proof}

\begin{lemma}\label{dist_tub}
Let $N_t$ be a tubular neighborhood of a smooth $k$-dimensional mean curvature flow $M_t$ then, $r(x,t)$ be the distance function from $M_t$ and $v_1(x,t),\ldots, v_k(x,t)$ be the principal directions of $M_t$ at $\pi(x,t) \in M_t$ w.r.t. the normal $-\nabla r(x)$ (where $\pi(x,t)$ is the nearest point projection to $M_t$). Then   
\begin{equation}
\frac{d}{dt}r-F(\nabla r,\nabla^2 r)=\left(\sum_{i=1}^{k}\frac{\left<A(v_i,v_i),-\nabla  r\right>^2}{1+r\left<A(v_i,v_i),-\nabla r\right>} \right)r
\end{equation}
on $N_t-M_t$.
\end{lemma}
\begin{proof}
Fix $t$, $x\in N_t-M_t$ and let $p=\pi(x,t)$. By the definition of $F$ and the fact that $||\nabla r||=1$, $F(\nabla r(x,t),\nabla^2r(x,t))$ is the sum of the $k$ smallest principal curvatures of $M^r_t$ at $x$. By Lemma \ref{dist_tub_stat}, those principal curvatures correspond to vectors in $T_{\pi(x,t)}M_t$, so since the trace of a matrix is independent of the basis we get, again by Lemma \ref{dist_tub_stat},  
\begin{equation}
F(\nabla r,\nabla^2 r)=\sum_{i=1}^{k}\frac{\left<A(v_i,v_i),-\nabla  r\right>}{1+r\left<A(v_i,v_i),-\nabla r\right>}. 
\end{equation} 
On the other hand, since $M_t$ moves by mean curvature, by the first variation of length we get
\begin{equation}
\frac{d}{dt}r(x,t)=\left<\vec{H}(\pi(x,t),t),-\nabla r\right>=\sum_{i=1}^{k}\left<A(v_i,v_i),-\nabla  r\right>.
\end{equation}
The result follows.
\end{proof}

\vspace{5 mm}

The following lower bound on the solution of the level set flow starting from a distance function from a set will be used in the proof of Theorem \ref{smooth_uniq}. As discussed above, it is based on the fact that although arbitrary co-dimensional mean curvature flow does not satisfy avoidance, it does satisfy it w.r.t. (co-dimension one) balls, moving according to the sum of their lowest $k$ principal curvatures (which for balls are, of course, the same).  For Brakke flows, this fact was already observed in Brakke's original manuscript \cite{Bra}. In a sense, the following lemma shows it for the Ambrosio-Soner level-set flow.      

\begin{lemma}\label{lower_bound}
Let $X$ be a closed set and let $p$ be a point with $d(X,p)=R$. Let $g$ be the distance function from $X$  and consider the level set flow $u$ starting from $g$. If $R^2>2kt$ then 
\begin{equation}
u(p,t) \geq R-\sqrt{2kt}
\end{equation} 
\end{lemma}
\begin{proof}
Let $\tilde{g}(y)$ be a function which equals $\min\{d(X,y),R-\sqrt{2kt}\}$ on $\mathbb{R}^{n}-B(p,\sqrt{2kt})$  and $R-d(p,y)$ on $B(p,\sqrt{2kt})$. Letting $\tilde{u}$ be the solutions to the level-set equation corresponding to $\tilde{g}$, for every $0\leq c < R-\sqrt{2kt}$, $\{x\;\;|\;\; u(x,t)=c\}=\{x\;\;|\;\; \tilde{u}(x,t)=c\}$ while by continuity and by the known evolution of spheres $\tilde{u}(p,t)=R-\sqrt{2kt}$. Thus $u(p,t)\geq R-\sqrt{2kt}$.    
\end{proof}

\subsection{Conclusion}\label{smooth_uniq_sec}
\begin{proof}[Proof of Thm. \ref{smooth_uniq}]
Set $X_0=X$, consider the functions $r(x,t)=\mathrm{dist}(x,X_t)$ and $v(x,t)=r^2(x,t)/\sqrt{t}$ and let $N=\{(x,t)\in \mathbb{R}^{n}\times[0,T]\;\;| r(x,t) < \frac{\sqrt{t}}{4c_1}\}$ and $N_t=N\cap \left(\mathbb{R}^n\times \{t\}\right)$. $v$ is smooth in $N$ and by Lemma \ref{dist_tub} 
\begin{equation}
\frac{d}{dt}v-F(\nabla v,\nabla^2 v)=2\left(\sum_{i=1}^{k}\frac{\left<A(v_i,v_i),-\nabla  r\right>^2}{1+r\left<A(v_i,v_i),-\nabla r\right>}-\frac{1}{4t} \right)v
\end{equation}
where $v_i$ are the principal directions of $X_t$ w.r.t. the normal $-\nabla r(x)$. Furthermore by the assumptions of the Theorem,  for every $(x,t)\in N$ we have 
\begin{equation}
\left|r\left<A(v_i,v_i),-\nabla r\right>\right|\leq \frac{c_1}{\sqrt{t}}\cdot \frac{\sqrt{t}}{4c_1}< \frac{1}{2},\;\;\;\;\;\;\;\;\;\;\;\sum_{i=1}^{k}\left<A(v_i,v_i),-\nabla  r\right>^2 \leq \frac{c_1^2}{t},
\end{equation}
and since $c_1^2 \leq 1/8$ we find that $v$ is a sub-solution to equation \eqref{visc_eq} in $N$. 

Let $d(x)$ be the distance function from $X=X_0$ and let $u$ be the solution to \eqref{visc_eq} with the initial data $d$. For $(x,t)\in \partial_{\mathrm{par}} N$ we have 
\begin{equation}
d(x) \geq r(x,t)-d_H(X,X_t) \geq \left(\frac{1}{4c_1}-c_2\right)\sqrt{t}>\sqrt{2kt}  
\end{equation}
so by Lemma \ref{lower_bound}
\begin{equation}
u(x,t) \geq \left(\frac{1}{4c_1}-c_2-\sqrt{2k}\right)\sqrt{t}\geq \alpha v(x,t)
\end{equation}
for $\alpha=16c_1^2\left(\frac{1}{4c_1}-c_2-\sqrt{2k}\right)>0$. Thus, by Property C, $u \geq \alpha v$ on $N$. The first part of the above argument also shows that $u>0$ outside of $N$. In particular, as the level set flow of $X$, $\tilde{X}_t$ is defined to be the zero set of $u$, and as $u>0$ on $\mathbb{R}^n-N_t$ and $u\geq \alpha v>0$ on $N_t-X_t$ we see that $\tilde{X}_t \subseteq X_t$.

The inclusion $X_t \subseteq \tilde{X}_t$ is is simpler (and in the co-dimension one case, follows immediately from Ilmanen's definition). Suppose $u(x_0,t_0)=\delta > 0$ for some $0\leq t_0 \leq T$ and $x_0 \in X_{t_0}$. For every $s<t_0$, the  distance function from $X_s$, $d^s(x)=r(x,s)$ satisfies $||d^s-d||_{L^{\infty}(\mathbb{R}^n)} \leq c_2\sqrt{s}$. Denote by $u^s$ the solution for the level set equation emanating from $d^s$. Then 
\begin{equation}
0<\delta=u(x_0,t_0) \leq |u(x_0,t_0)-u(x_0,t_0-s)|+|u(x_0,t_0-s)-u^s(x_0,t_0-s)|
\end{equation}  
where we have used the fact that $u^s(x_0,t_0-s)=0$ by Property A. The first term of the left hand side goes to zero as $s\rightarrow 0$ since $u$ is continuous, while the second term also goes to zero since $||u^s-u||_{L^{\infty}(\mathbb{R}^n\times[0,T])} \leq ||d^s-d||_{L^{\infty}(\mathbb{R}^n)}$ by Property B. This is a contradiction, so if $x_0\in X_{t_0}$ we must have $u(x_0,t_0)=0$, i.e. $x_0\in \tilde{X}_{t_0}$.
\end{proof}

\bibliographystyle{alpha}
\bibliography{ReifBib}

\vspace{10mm}
{\sc Or Hershkovits, Courant Institute of Mathematical Sciences, New York University, 251 Mercer Street, New York, NY 10012, USA}\\

\emph{E-mail:}  or.hershkovits@cims.nyu.edu

\end{document}